\documentclass[11pt,oneside,reqno]{article}

\usepackage{amsmath,amsthm,amssymb,color}
\usepackage{bbm}
\usepackage{mathrsfs}
\usepackage[breaklinks=true]{hyperref}
\usepackage{textcase}

\numberwithin{equation}{section}
\allowdisplaybreaks[4]

\theoremstyle{plain}
\newtheorem{theorem}{Theorem}[section]

\newtheorem{proposition}[theorem]{Proposition}
\newtheorem{lemma}[theorem]{Lemma}

\newtheorem{conjecture}[theorem]{Conjecture}

\theoremstyle{definition}
\newtheorem{definition}{Definition}[section]

\newtheorem{remark}[definition]{Remark}

\makeatletter

\def\^#1{\ifmmode {\mathaccent"705E #1} \else {\accent94 #1} \fi}
\def\~#1{\ifmmode {\mathaccent"707E #1} \else {\accent"7E #1} \fi}

\def\*#1{#1^\ast}
\edef\-#1{\noexpand\ifmmode {\noexpand\bar{#1}} \noexpand\else \-#1\noexpand\fi}
\def\>#1{\vec{#1}}
\def\.#1{\dot{#1}}

\def\atop{\@@atop}
\def\%#1{\mathcal{#1}}

\renewcommand{\leq}{\leqslant}
\renewcommand{\geq}{\geqslant}
\renewcommand{\phi}{\varphi}

\newcommand{\eq}{\eqref}

\def\tsfrac#1#2{{\textstyle\frac{#1}{#2}}}

\newcommand{\IE}{\mathbbm{E}}

\newcommand{\IR}{\mathbbm{R}}

\def\ba#1{\begin{align*}#1\end{align*}}
\def\ban#1{\begin{align}#1\end{align}}

\def\norm#1{\Vert#1\Vert}

\def\abs#1{\vert#1\vert}

\newcount\minute
\newcount\hour
\newcount\hourMins
\def\now{%
\minute=\time%
\hour=\time \divide \hour by 60%
\hourMins=\hour \multiply\hourMins by 60%
\advance\minute by -\hourMins%
\zeroPadTwo{\the\hour}:\zeroPadTwo{\the\minute}%
}
\def\zeroPadTwo#1{\ifnum #1<10 0\fi#1}

\renewcommand\section{\@startsection {section}{1}{\z@}%
{-3.5ex \@plus -1ex \@minus -.2ex}%
{1.3ex \@plus.2ex}%
{\center\small\sc\MakeTextUppercase}}

\def\subsection#1{\@startsection {subsection}{2}{0pt}%
{-3.5ex \@plus -1ex \@minus -.2ex}%
{1ex \@plus.2ex}%
{\bf\mathversion{bold}}{#1}}

\def\subsubsection#1{\@startsection{subsubsection}{3}{0pt}%
{\medskipamount}%
{-10pt}%
{\normalsize\itshape}{\kern-2.2ex. #1.}}

\def\blfootnote{\xdef\@thefnmark{}\@footnotetext}

\makeatother

\def\ed{\stackrel{d}{=}}
\def\t#1{^{(#1)}}

\usepackage{graphicx}
\begin{document}

\title{\sc\bf\large\MakeUppercase{
Archimedes, Gauss, and Stein}}
\author{Jim Pitman\thanks{Research supported in part by NSF grant DMS-0806118}\,\, and Nathan Ross}
\date{\it University~of~California,~Berkeley}
\maketitle

\begin{abstract}
We discuss a characterization of the centered Gaussian distribution
which can be read from results of Archimedes and Maxwell,
and relate it to Charles Stein's well-known characterization of the same distribution.
These characterizations fit into a more general framework involving
the beta-gamma algebra, which explains
some other characterizations appearing
in the Stein's method literature.
\end{abstract}

\section{Characterizing the Gaussian distribution}

One of Archimedes' proudest accomplishments 
was a proof that the surface area of a sphere is 
equal to the surface area of the tube of the smallest cylinder 
containing it; see Figure \ref{fig1}.  
Legend has it that he was so pleased 
with this result that he arranged to have
an image similar to Figure \ref{fig1} inscribed on his tomb.
\begin{figure}[h]
\begin{center}
\includegraphics[scale=.5]{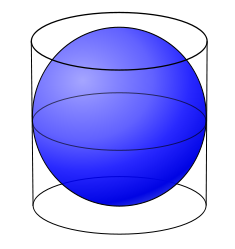}
\end{center}
\caption{An illustration of the inscription on Archimedes' tomb.} \label{fig1}
\end{figure}

More precisely, in the work ``On the Sphere and Cylinder, Book I" as translated on Page 1 of \cite{hea97},
Archimedes states that for every plane perpendicular
to the axis of the tube of the cylinder, the surface areas lying above the plane on the sphere and on the tube are equal.
See Figure \ref{fig2} for illustration and also the discussion around Corollary 7 of \cite{apmn04}.

\begin{figure}[h]
\begin{center}
\includegraphics[scale=.5]{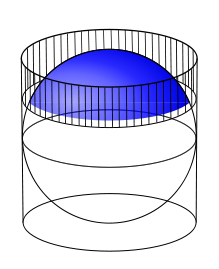}
\end{center}
\caption{The surface area of the shaded ``cap" of the sphere above a plane
is equal to the striped surface area on the tube of the cylinder above the same plane.} \label{fig2}
\end{figure}

In probabilistic terms, if a point is picked uniformly at random according to surface area
on the unit sphere in three dimensions, then its projection onto any given axis having origin at the center of the sphere
is uniformly distributed on the interval $(-1,1)$, independent of the angular
component in the plane perpendicular to that axis.  Formally, we have the following result.

\begin{proposition}\label{proparc}
If $V$ is uniformly distributed on the interval $(-1,1)$ and $\Theta$ is
uniformly distributed on the interval $(0,2\pi)$ and is independent 
of $V$, then
\ba{
\left(V, \sqrt{1-V^2}  \cos(\Theta), \sqrt{1-V^2} \sin(\Theta) \right) 
}
is uniformly distributed on the surface of the two dimensional sphere of radius one.
\end{proposition}

In this article, we take Proposition \ref{proparc} as a starting point
for a discussion of characterizations of the centered Gaussian distribution
which arise in Stein's method of distributional approximation.
This discussion culminates in Theorem \ref{thmnorm} at the end of this section.
We then generalize some of these results in Section \ref{bgasec}
to obtain the characterization of the gamma distribution found 
in Proposition \ref{prop2a}, and also mention an analog of Theorem \ref{thmnorm} for the exponential distribution.
We conclude in Section \ref{secf} with a discussion of some related literature.

To move from Archimedes' result above to 
characterizing the Gaussian distribution, we 
state the following result which
was first realized by the astronomer Herschel and 
made well known by the physicist Maxwell in his 
study of the velocities of a large number of gas particles in  
a container; see the introduction of \cite{bry95}. 
\begin{proposition}\label{propmax}
Let $\textbf{X}=(X_1, X_2, X_3)$ be a vector of independent and identically 
distributed (i.i.d.) 
random variables. Then 
$X_1$ has a mean zero Gaussian distribution if and only if for all 
rotations $R:\IR^3\to\IR^3$, $R\textbf{X}$ has the same distribution as $\textbf{X}$.
\end{proposition}

Propositions \ref{proparc} and \ref{propmax} are
related by the following observations.
It is clear that if $\textbf{X}$ is an $\IR^3/\{0\}$ valued
random vector such that $R \textbf{X}$ has the same distribution as $\textbf{X}$ for 
all rotations $R$, then $\textbf{X}/\norm{\textbf{X}}$ 
is a rotation invariant distribution on the surface of the two dimensional unit sphere 
and is independent of $\norm{\textbf{X}}:=\sqrt{X_1^2+X_2^2+X_3^2}$.
Since the unique rotation invariant distribution on the surface of a sphere of any dimension
is the uniform distribution (Theorem 4.1.2 of \cite{bry95}), the
propositions of Archimedes and Herschel-Maxwell
suggest the following characterization of mean zero Gaussian distributions; we provide a proof and discussion
of generalizations in the Appendix.
\begin{proposition}\label{proparcmax}
Let $\textbf{X}=(X_1, X_2, X_3)$ be a vector of i.i.d.\ random variables. Then
$X_1$ has a mean zero Gaussian distribution if and only if for $V$ uniform on $(-1,1)$ and independent of $\textbf{X}$,
\ba{
X_1\ed V\sqrt{X_1^2+X_2^2+X_3^2}.
}
\end{proposition}
Here and in what follows, $\ed$ denotes equality in distribution of two random variables.
The distribution of $\sqrt{X_1^2+X_2^2+X_3^2}$, where $X_1, X_2, X_3$  are independent standard normal variables,
is referred to as the Maxwell or Maxwell-Boltzman distribution; see page 453 of \cite{jkb94}.

Proposition \ref{proparcmax} characterizes
centered Gaussian distributions as the one parameter scale family
of fixed points of the distributional transformation
which takes the distribution of a random variable $X$ 
to the distribution of $V\sqrt{X_1^2+X_2^2+X_3^2}$, where $X_1, X_2, X_3$ are i.i.d.\ copies of $X$, and
$V$ is uniform on $(-1,1)$ independent of $(X_1, X_2, X_3)$.  Such characterizations
of distributions as the unique fixed point of a transformation are often
used in Stein's method for distributional approximation (see \cite{ros11} for an introduction).
In the case of the Gaussian distribution, these transformations are put
to use
through Stein's Lemma.
\begin{lemma}[Stein's Lemma]\cite{stn72}\label{lemstein}
A random variable $W$ has the mean zero, variance one Gaussian distribution 
if and only if for all absolutely continuous functions $f$ with bounded derivative,
\ba{
\IE f'(W)=\IE Wf(W). 
}
\end{lemma}

We can relate the characterizations provided by Proposition \ref{proparcmax}
and Lemma \ref{lemstein}, but first we need the following definition.

\begin{definition}
Let $X$ be a random variable with distribution function $F$ and such that $\mu_\alpha:=\IE \abs{X}^\alpha <\infty$.  
We define $F^{(\alpha)}$, the 
$\alpha$-power bias distribution of $F$, by the relation
\ba{
dF^{(\alpha)}(x)=\frac{\abs{x}^\alpha dF(x)}{\mu_\alpha},
}
and we write $X\t\alpha$ for a random variable having this distribution.
Otherwise put, $X\t\alpha$ has
the $\alpha$-power bias distribution of $X$ if and only if 
for every measurable function $f$ such that $\IE \abs{X}^\alpha \abs{f(X)} <\infty$, 
\ban{
 \IE f(X\t\alpha)= \frac{\IE \abs{X}^\alpha f(X)}{ \IE \abs{X}^\alpha}. \label{propbias} 
}
\end{definition}
Taking $\alpha=1$ and $X\geq0$, $X\t1$ has the \emph{size-biased} distribution
of $X$, a notion which frequently arises in probability theory
and applications \cite{argo11, bro06}.

We can now state and prove the following result which sheds some light on the 
relationship between Proposition \ref{proparcmax} and Lemma \ref{lemstein}.
\begin{lemma}\label{lem1s}
If $W$ is a random variable with finite second moment and
$f$ is an absolutely continuous function with bounded derivative, then
for $V$ uniform on the interval $(-1,1)$ and independent of $W$,
\ban{
2\IE W^2 \IE f'(VW\t2)=\IE Wf(W)-\IE Wf(-W). \label{s1}
}
\end{lemma}
\begin{proof}
The lemma is implied by the following calculation
\ba{
\IE f'(VW\t2)&=\frac{1}{2}\IE \left[\int_{-1}^1f'(uW\t2) du\right]\\
	&=\frac{1}{2}\IE \left[\frac{f(W\t2)-f(-W\t2)}{W\t2}\right]\\
	&=\frac{\IE Wf(W)- Wf(-W)}{2\IE W^2},	
}
where in the final equality we use \eqref{propbias}.
\end{proof}

We now have the following main result for the Gaussian distribution.

\begin{theorem}\label{thmnorm}
Let $W$ be a random variable with finite second moment.
The following are equivalent:
\begin{enumerate}
\item $W$ has the standard normal distribution.
\item For all absolutely continuous functions $f$ with bounded derivative,
\ban{
\IE f'(W)=\IE Wf(W). \label{1}
}
\item $\IE W^2 =1 $ and $W\ed VW^{(2)}$, where 
V is uniform on $(-1,1)$ and independent of $W^{(2)}$.
\end{enumerate}
\end{theorem}
\begin{proof}
The equivalence of the first two items of the proposition is (Stein's) Lemma \ref{lemstein} above.  

The fact that Item 1
implies Item 3 follows from Proposition \ref{proparcmax} above coupled with the simple
fact that
for $X_1, X_2, X_3$ i.i.d.\ standard normal random variables,
the density of $(X_1^2+X_2^2+X_3^2)^{1/2}$
is proportional to $x^2 e^{-x^2/2}$ (that is, $(X_1^2+X_2^2+X_3^2)^{1/2}$ has the same distribution as $X_1\t2$).

Finally, we show Item 2 follows from Item 3.
If $W\ed VW\t2$ and $\IE W^2=1$, then using Lemma \ref{lem1s} we find that for functions $f$ with bounded derivative,
\ba{
\IE f'(W)= \IE f'(VW\t2)=\frac{1}{2}\left(\IE Wf(W)-\IE Wf(-W)\right)=\IE Wf(W),
}
where the last equality follows from the assumptions of Item 3 which imply $W$ has the 
same distribution as $-W$.
\end{proof}

\begin{remark}
The equivalence of Items 1 and 3  
is essentially the content of Proposition 2.3 of \cite{cgs10}, which uses
the concept of the ``zero-bias" transformation of Stein's method, first introduced in \cite{gore97}.
For a random variable $W$ with mean zero and variance $\sigma^2<\infty$, we 
say that $W^*$ has the zero-bias distribution of $W$ if
for all $f$ with $\IE\abs{Wf(W)}<\infty$,
\ba{
\sigma^2\IE f'(W^*)= \IE Wf(W).
}
We think of the zero-bias transformation acting on probability distributions with zero mean and finite variance,
and Stein's Lemma implies that this transformation has the centered Gaussian distribution as its unique fixed point.
Proposition 2.3 of \cite{cgs10} 
states that for a random variable $W$ with support symmetric about zero with unit variance,
the transformation $W\to V W^{(2)}$ provides a representation of the zero-bias transformation.
The equivalence of Items~1 and~3 of the theorem follows easily from these results.
\end{remark}

\section{Beta-Gamma Algebra}\label{bgasec}
The equivalence
between Items 1 and 3 in Theorem \ref{thmnorm} 
can be generalized as follows. For $r,s>0$, let $G_r$, and $B_{r,s}$ denote standard gamma and beta 
random
variables having respective densities $\tsfrac{1}{\Gamma(r)}x^{r-1}e^{-x}, x>0$ and
$\tsfrac{\Gamma(r+s)}{\Gamma(r)\Gamma(s)}y^{r-1}(1-y)^{s-1}, 0<y<1$, where $\Gamma$ denotes the gamma function.

\begin{proposition}\label{prop2a}
Fix $p,r,s>0$.
A non-negative random variable $W$ has the distribution 
of $c \,G_r^p$ for some constant $c>0$ if and only if
$W\ed B_{r,s}^p W^{(s/p)}$,
where $B_{r,s}$ is independent of $W\t{s/p}$.
\end{proposition}

\begin{remark}
The equivalence in Items 1 and 3 of Theorem \ref{thmnorm}  follows
by taking  $p=r=1/2, s=1$ in Proposition \ref{prop2a}
and using the well known fact that for $Z$ having the standard
normal distribution, $Z^2 \ed 2 G_{1/2}$.
\end{remark}

The proof of Proposition \ref{prop2a} uses the following result.

\begin{lemma}\label{bipo}
Let $\alpha, \beta>0$.  If $X\geq0$ is a random variable such that 
$\IE X^\alpha<\infty$, then
\ba{
(X\t\alpha)^\beta\ed (X^\beta)\t{\alpha/\beta}.
} 
\end{lemma}
\begin{proof}
By the definition of $\alpha/\beta$-power biasing, we only need to show that
\ban{
\IE X^\alpha \IE f((X\t\alpha)^\beta)= \IE X^\alpha f(X^\beta) \label{bi1}
}
for all $f$ such that the expectation on the left hand side exists.  By the definition of
$\alpha$-power biasing, we have that for $g(t)=f(t^\beta)$,
\ba{
\IE X^\alpha \IE g(X\t\alpha) = \IE X^\alpha g(X),
}
which is \eq{bi1}.
\end{proof}

\begin{proof}[Proof of Proposition \ref{prop2a}]
The usual beta-gamma algebra (see \cite{duf98}) implies that
$G_r \ed B_{r,s} G_{r+s}$ where $B_{r,s}$ and $G_{r+s}$ are independent.
Using the elementary fact that
$G_{r+s} \ed G_r^{(s)}$, we
find that for fixed $r,s > 0$, $G_r$ satisfies
$G_r \ed B_{r,s} G_r^{(s)}$. 
Now applying
Lemma \ref{bipo} to $G_r$ with $\alpha=s$ and $\beta=p$,
we have that $W = G_r^p$ satisfies
$W \ed B_{r,s}^p W\t{s/p}$ and the forward implication now follows after noting that $(cX)\t\alpha\ed cX\t\alpha$

Now, assume that $W\ed  B_{r,s}^p W ^{(s/p)}$ for fixed $p,r,s>0$ and we show that $W\ed c\,G_r^p$ for some $c>0$.
First, note by Lemma \ref{bipo}, that if $X=W^{1/p}$, then 
\ban{
X\ed B_{r,s} X\t{s} \label{r1}
}
and we will be done if this implies that $X\ed G_r$.  
Note that by writing $X\t{s}$, we have been tacitly assuming that $  \IE W^{s/p}=\IE X^s<\infty$, which implies
that $\IE (B_{r,s} X\t{s})^s<\infty$ so that using the definition of power biasing yields
$\IE X^{2s} <\infty$.  Continuing in this way we find
that $\IE X^{ks} <\infty$ for all $k=1,2,\ldots$ and thus that $\IE X^p <\infty$ for all $p\geq s$.
Moreover, writing $a_k:=\IE X^{ks}$, and taking expectations in \eq{r1} after raising both sides to the power $k$, we have
\ba{
a_{k}=\IE B_{r,s}^{ks} \frac{a_{k+1}}{a_1},
}
where we have again used the definition of power biasing.  We can solve this recursion
after noting that for $\alpha>-r$,
\ba{
\IE B_{r,s}^{\alpha}=\frac{\Gamma(r+\alpha)\Gamma(r+s)}{\Gamma(r+\alpha+s)\Gamma(r)},
}
to find that for $k=0,1,\ldots,$ 
\ba{
a_{k}=\left(\frac{a_1\Gamma(r)}{\Gamma(r+s)}\right)^k \frac{\Gamma(r+sk)}{\Gamma(r)}.
}
For any value of $a_1>0$, it is easy to see using Stirling's formula that the sequence $(a_k)_{k\geq1}$ satisfies Carleman's condition
\ba{
\sum_{k=1}^n a_{2k}^{-1/2k}  \to \infty, \mbox{\, as \,} n\to\infty,
}
so that for a given value of $a_1$, there is exactly one probability distribution having
moment sequence $(a_k)_{k\geq1}$ (see the remark following Theorem (3.11) in Chapter 2 of \cite{dur96}).
Finally, it is easy to see that the random variable
\ba{
X^s:= \frac{a_1\Gamma(r)}{\Gamma(r+s)} G_r^s
}
has moment sequence $(a_k)_{k\geq1}$.
\end{proof}

\subsection{Exponential Distribution}
The exponential
distribution has many characterizing properties, many of which stem from its relation to Poisson processes.
For example, by
superimposing two independent Poisson processes into one, we easily find that
if $Z_1$ and $Z_2$ are independent rate one exponential variables,
then 2$\min\{Z_1, Z_2\}$ is also a rate one exponential (this is in fact characterizing as shown in Theorem 3.4.1 of \cite{bry95}). 

For our framework above, we use the memoryless property of the exponential distribution
in the context of renewal theory.  
In greater detail, for any non-negative random variable $X$, we define the
renewal sequence generated from $X$ as $(S_1, S_2, \ldots)$, where $S_i=\sum_{k=1}^i X_k$ and
the $X_k$ are i.i.d.\ copies of $X$.   For a fixed $t>0$, the distribution of the
length of the interval $[S_{K_t}, S_{K_t+1}]$ containing $t$ and the position of $t$ in this interval
depend on $t$
and the distribution of $X$ in some rather complicated
way.  
We can remove this dependence on $t$ by starting the sequence in ``stationary" meaning
that we look instead at the sequence $(X', X'+S_1, \ldots)$, where $X'$ has 
the limiting distribution of $S_{K_t+1}-t$ as $t$ goes to infinity; see Chapter 5, Sections 6 and 7.b of \cite{kata75}.

If $X$ is a continuous distribution with finite mean, then the 
distribution of $X'$ is the size-biased distribution of $X$ times an independent variable which is uniform on $(0,1)$ \cite{kata75}.
Heuristically, the memoryless property which characterizes the exponential distribution (Chapter 12 of \cite{baba95})
implies that the renewal sequence generated by an exponential distribution is stationary (that is, $X$ and $X'$ have the
same distribution) and vice versa.  The
following result implies this intuition is correct.

\begin{theorem}\label{thmexp}\cite{rope10}
Let $W$ be a non-negative random variable with finite mean.
The following are equivalent:
\begin{enumerate}
\item $W$ has the exponential distribution with mean one.
\item For all absolutely continuous functions $f$ with bounded derivative,
\ba{
\IE f'(W)=\IE f(W)-f(0). 
}
\item $\IE W =1 $ and $W\ed UW\t1$, where 
U is uniform on $(0,1)$ and independent of $W\t1$.
\end{enumerate}
\end{theorem}
Similar to the case of the normal distribution, 
the crucial link between Items 2 and 3 of Theorem \ref{thmexp} is provided by the following
lemma; the proof is similar to that of Lemma \ref{lem1s}.
\begin{lemma}\label{lem2s}
If $W$ is a non-negative random variable with finite mean and
$f$ is an absolutely continuous function with bounded derivative, then
\ba{
\IE W\IE f'(UW\t1)=\IE f(W)-f(0). 
}
\end{lemma}

\begin{proof}[Proof of Theorem \ref{thmexp}]
The equivalence of Items 1 and 3 is a special
case of Theorem \ref{thmexp} with $r=s=p=1$,
and the equivalence of Items 2 and 3 can be read from 
Lemma \ref{lem2s} (note in particular that Item 2 with $f(x)=1$ implies
that $\IE W=1$).
\end{proof}

\begin{remark}
For a non-negative random variable $W$ with finite mean,
the transformation $W\to U W^{(1)}$ is referred to
in the Stein's method literature as the ``equilibrium" transformation,
first defined in this context in \cite{rope10}, where Theorem \ref{thmexp} is
also shown.

Due to the close relationship between the exponential and geometric distributions, 
it is not surprising that there is a discrete analog of Theorem \ref{thmexp} with the
exponential distribution replaced by the geometric; see \cite{prr10} for this discussion in
the context of Stein's method.
\end{remark}

\section{Proof of Proposition \ref{proparcmax} and discussion}\label{secf}

\begin{proof}[Proof of Proposition \ref{proparcmax}]
We will show that for $n\geq2$ and $Y_1,\ldots, Y_n$ non-negative i.i.d.\ random variables,
$Y_1\ed  c G_{1/(n-1)}$ for some $c>0$ if and only if
\ban{
Y_1\ed B_{1/(n-1),1} (Y_1+\cdots + Y_n), \label{bg1g}
}
where $B_{1/(n-1),1}$ is independent of $(Y_1,\ldots, Y_n)$, and $G_{a}, B_{a,b}$ are gamma and
beta variables as defined above.  The proposition then follows from this fact with $n=3$ 
after noting that $V^2\ed B_{1/2,1}$
and if $X$ has a mean zero and variance one normal distribution, then $X^2\ed 2G_{1/2}$.

The forward implication is a consequence of Proposition \ref{prop2a}
coupled with the fact that $G_{a+b}\ed G_a+G_b$, where $G_a$ and $G_b$ are independent.
To establish the result
we assume~\eq{bg1g} and show $Y_1\ed c G_{1/(n-1)}$.  Since we assume that $Y_1$ is non-negative,
we define the Laplace transform 
\ba{
\phi(\lambda)=\IE e^{-\lambda Y_1}, \lambda\geq0.
}
By conditioning on the value of $B_{1/(n-1),1}$ in \eq{bg1g}, we find for $\lambda>0$,
\ba{
\phi(\lambda)&=\IE \phi(B_{1/(n-1),1} \lambda)^n \\
	&=\frac{1}{n-1}\int_0^1u^{-(n-2)/(n-1)}\phi(u\lambda)^n du\\
	&=\frac{1}{(n-1)\lambda^{1/(n-1)}}\int_0^\lambda t^{-(n-2)/(n-1)} \phi(t)^n dt,
}
where we have made the change of variable $t=u \lambda$ in the last equality.  
We can differentiate the equation above with respect to $\lambda$ which yields
\ban{
\phi'(\lambda)&=-\frac{\lambda^{-n/(n-1)}}{(n-1)^2}\int_0^\lambda t^{-(n-2)/(n-1)} \phi(t)^n dt+\frac{1}{(n-1)\lambda} \phi(\lambda)^n, \notag \\
	\phi'(\lambda)&=\frac{-\phi(\lambda)+\phi(\lambda)^n}{(n-1)\lambda}. \label{odeg}
}
Thus, we find that $\phi$ satisfies the differential equation \eq{odeg} with boundary condition $\phi(0)=1$.

By computing the derivative using \eq{odeg} and using that $0<\phi(\lambda)\leq 1$ for $\lambda>0$,  
we find that for some constant $c\geq0$,
\ba{
\frac{1-\phi(\lambda)^{n-1}}{\lambda\phi(\lambda)^{n-1}}=c, \,\,\,  \lambda>0.
}
Solving this equation for $\phi(\lambda)$ implies 
\ba{
\phi(\lambda)=(1+c \lambda)^{-1/(n-1)},
}
which is the Laplace transform of $c G_{1/(n-1)}$, as desired.

\end{proof}

The proof of Proposition~\ref{proparcmax} and the beta-gamma algebra
suggest the following conjecture.
\begin{conjecture}\label{c1}
Let $n\geq2$ and $\textbf{Y}=(Y_1, Y_2,\ldots, Y_n)$ be a vector of i.i.d.\ random variables. Then
$Y_1$ is equal in distribution to $cG_{a}$ for some constant $c>0$ if and only if for $V=B_{a,(n-1)a}$ independent of $\textbf{Y}$,
\ban{
Y_1\ed V(Y_1+Y_2+\ldots +Y_n). \label{bg22a}
}
\end{conjecture}

The forward implication of the conjecture is an easy consequence of the following beta-gamma algebra facts:
for $G_a$, $G_b$, and $B_{a,b}$ independent, $B_{a,b}G_{a+b}\ed G_a$, and $G_a+G_b\ed G_{a+b}$. 

Conversely, assuming \eq{bg22a}, it is possible to follow the proof of Proposition \ref{proparcmax}, which leads to an integral equation
for the Laplace transform of $Y_1$. It is easy to verify that the Laplace transform of the appropriate gamma distribution satisfies this
equation, so it is only a matter of showing the integral equation has a unique scale family of solutions.  
In the case $a=1/(n-1)$ the integral equation has a simpler form from which the required uniqueness follows from
the proof of Proposition \ref{proparcmax} above.  
In the 
general case, we do not have an argument for the uniqueness of the solution.
However, under the assumption that $Y_1$ has all positive integer moments finite, the conjecture follows after using
\eq{bg22a} to obtain
a recursion relation for the moments which, up to the scale factor, determines those of a gamma distribution with the appropriate parameter.

Conjecture \ref{c1} is very similar to Lukacs' characterization of the the gamma distribution \cite{luk55}
that
positive, non-degenerate, independent variables $X,Y$ have the gamma distribution if and only if $X+Y$ and $X/(X+Y)$ 
are independent.  However, it does not appear that this result can be used
to show the difficult implication of the conjecture.  Note also that Lukacs' result also
characterizes  beta distributions as the only distributions which can be 
written as $X/(X+Y)$ independent of $X+Y$ for positive, non-degenerate, independent variables $X,Y$.  Thus, a question
related to our conjecture is that if \eq{bg22a} holds for independent variables $Y_1,\ldots,Y_n$ and $V$, does this imply that $V$ has
a beta distribution?

Conjecture \ref{c1} is connected to the observation of Poincar\'{e} (see the introduction of \cite{mck73})
that the coordinates of a point uniformly chosen on the $(n-1)$ dimensional sphere of radius $\sqrt{n}$ 
are asymptotically distributed as independent standard Gaussians.  Analogous to the discussion in the introduction, we can
realize these uniformly distributed points as $\sqrt{n}R^{-1}(X_1, \ldots, X_n)$, where 
$X_1, \ldots, X_n$ are
independent standard normal variables and $R=(X_1^2+\cdots+X_n^2)^{1/2}$.
Squaring these coordinates, Poincar\'e's result 
implies that $nX_1^2/(X_1^2+\cdots+X_n^2)$ is asymptotically distributed as $X_1^2$.
Since $X_1^2\ed 2G_{1/2}$, taking the limit as $n\to\infty$ on the right side of \eq{bg22a} with $a=1/2$
yields a related fact.

The forward implication of Proposition \ref{proparcmax} 
is evidenced also
by creation of a three-dimensional Bessel process by conditioning a one-dimensional
Brownian motion not to hit zero. Indeed, a process
version of 
Proposition \ref{proparcmax} is involved 
in the proof of the ``$2M-X$" theorem provided in \cite{piro81}; see
Section 2 and especially Section 2.3 of \cite{ppy98}.
More generally, process analogs of the beta-gamma
algebra
can be found in Section 3 of \cite{ppy98}.  

Some extensions 
of the characterizations discussed in this article to
more complicated distributions can be
found in the recent work \cite{prr11}.

\def\polhk#1{\setbox0=\hbox{#1}{\ooalign{\hidewidth
  \lower1.5ex\hbox{`}\hidewidth\crcr\unhbox0}}}

\end{document}